\documentclass[11pt,a4paper]{article}
\usepackage[english]{babel}        \usepackage[latin1]{inputenc}
\usepackage[T1]{fontenc}          \usepackage{a4}
\usepackage[cyr]{aeguill}            \usepackage{epsfig}
\usepackage{amsmath,amsthm,amsfonts,amssymb}
\usepackage{dsfont}                  \usepackage{float}
\usepackage{fancyhdr}               \usepackage{varioref}
\usepackage{url}                        \usepackage{calc}
\usepackage{wrapfig}                 \usepackage{graphicx}
\newtheorem{theorem}{Theorem}
\newtheorem{definition}{Definition}[section]
\newtheorem{remark}{Remark}[section]
\newtheorem{proposition}{Proposition}[section]
\newtheorem{lemma}{Lemma}[section]
\newtheorem{corollary}{Corollary}[section]
\numberwithin{equation}{section}

\usepackage{authblk}

\makeatletter
\def\namedlabel#1#2{\begingroup
    #2%
    \def\@currentlabel{#2}%
    \phantomsection\label{#1}\endgroup
}
\makeatother

\usepackage{enumitem,hyperref}
\hypersetup{
backref=true, 
pagebackref=true,
hyperindex=true, 
colorlinks=true, 
breaklinks=true, 
urlcolor= blue, 
linkcolor= blue, 
bookmarks=true, 
bookmarksopen=true, 
}

\newcommand{\ensemblenombre}[1]{\mathbb{#1}}
\newcommand{\Esp}{\ensemblenombre{E}}
\newcommand{\Z}{\ensemblenombre{Z}}

\newcommand{\R}{\ensemblenombre{R}}

\newcommand{\Scb}{S}
\newcommand{\A}{\mathcal{A}}
\newcommand{\tomega}{\widetilde{\omega}}
\newcommand{\Ocb}{\Omega}
\newcommand{\Fcb}{\mathcal{F}}
\newcommand{\Sc}{\widetilde{S}}
\newcommand{\Oc}{\widetilde{\Omega}}
\newcommand{\Fc}{\widetilde{\mathcal{F}}}
\newcommand{\tpi}{\widetilde{\pi}}
\newcommand{\tP}{\widetilde{P}}
\newcommand{\tX}{\widetilde{X}}
\newcommand{\tm}{\widetilde{m}}
\newcommand{\1}{\mathds{1}}

\begin{document}
\title{Percolation results for the Continuum Random Cluster Model}

\author{Pierre Houdebert}
\affil[1]{Laboratoire de Math\'ematiques Paul Painlev\'e\\University of Lille 1, France 
 \texttt{pierre.houdebert@math.univ-lille1.fr}}

\maketitle

\begin{abstract}
{The continuum random cluster model is a Gibbs modification of the standard Boolean model with intensity $z>0$ and law of radii $Q$. 
The formal
unnormalized density is given by $q^{N_{cc}}$ where $q$ is a fixed parameter and $N_{cc}$ is the number of connected components in the random structure. 
We prove for a large class of parameters that percolation occurs for $z$ large enough and does not occur for $z$ small enough. 
An application to the phase transition of the Widom-Rowlinson model with random radii is given.
Our main tools are stochastic domination properties, a fine study of the interaction of the model and a Fortuin-Kasteleyn representation.
}
  \bigskip


  \bigskip

\end{abstract}

\section{Introduction}
This paper is interested in the Continuum Random Cluster Model (called CRCM in the following), which is defined on a bounded window as a penalized Poisson Boolean model with intensity $z>0$ and law of radii $Q$.
The unnormalized density is  $q^{N_{cc}}$ where $q>0$ is a positive real number and $N_{cc}$ denotes the number of connected components of the random closed set considered. 
In the infinite volume regime a global density is meaningless and a definition of the CRCM using Gibbs modifications is required. 
In this paper we prove the existence of a subcritical and supercritical phase for the percolation of the CRCM.

The original random cluster model was introduced as a lattice model in the late 1960's by Fortuin and Kasteleyn to unify models of percolation as Ising and Potts models. 
Properties and results about this model, such as existence of random cluster model on infinite graphs, percolation and phase transition properties can be found in \cite{georgii_haggstrom_maes,grimmett_livre_rcm}. 
In the continuum setting the CRCM  is directly linked to the well-known Widom-Rowlinson multi-type model by the so-called Fortuin-Kasteleyn representation (also known as grey representation or random cluster representation).
This representation was used in the 1990's to give a new proof of the phase transition for the Widom-Rowlinson model \cite{chayes_kotecky}, or for the more general class of continuum Potts models \cite{georgii_haggstrom}.
The CRCM is also studied in stochastic geometry and spatial statistics as an interacting random germ-grain model \cite{moller_helisova08,moller_helisova10}. 
For a suitable parameter $q$ the CRCM fits as best as possible the clustering of the real dataset. 


The existence of the infinite volume CRCM, defined throught standard DLR formalism, was recently investigated in \cite{dereudre_houdebert} which gives an existence result in cases like $q<1$ or unbounded radii.

Percolation refers to the existence of at least one unbounded (or infinite) connected component in the random structure.
For the Poisson Boolean model percolation is well-understood, see \cite{meester_roy}, and we have the existence of a positive  threshold $z^*$ for the intensity, such that
percolation occurs for $z>z^*$ and not for $z<z^*$.
Because of its strong independence properties, the Poisson Boolean model is sometimes irrelevant for applications in physics or biology.

In recent works \cite{jansen_2016,stucki_2013} percolation was proved to occur in the case of Boolean model with deterministic radii and germs driven by a Gibbs point process.
Both proofs strongly rely on the independent  marking of the Boolean model and don't apply for the CRCM.

Moreover, in statistical mechanics the percolation of the CRCM leads to the phase transition of the Widom-Rowlinson model.
This was already done in the case of deterministic radii in \cite{chayes_kotecky} and \cite{georgii_haggstrom}.
Percolation is also related to uniqueness of the Gibbs measure, with the method of disagreement percolation, which is based on constructing a coupling with good properties, see \cite{georgii_haggstrom_maes,vandenberg_maes}.

The aim of the present paper is to provide percolation results for the CRCM with parameters as general as possible. 
The two cases  considered are (C1) $q \geq 1$, $\int R^d Q(dR)<\infty$ and (C2) $q<1$, $Q([0,R_0])=1$ for some $R_0$.
In both cases we prove the absence of percolation for small $z$ and the existence of percolation for large $z$.
Part of the result is trivially obtained using a stochastic domination result of Georgii and K\"uneth \cite{georgii_kuneth}. 
The challenging part,
proving the existence of percolation for large $z$ in (C1) and the absence of percolation for small $z$ in (C2),
relies on a discrete stochastic domination result proved by Liggett, Schonmann and Stacey in \cite{liggett_schonmann_stacey}.
This result is stated in Proposition \ref{proposition_liggett_domination_sto}.
This method was already successfully used in \cite{coupier_dereudre} to prove a percolation result for the continuum Quermass-interaction model.

The paper is organized as follows. 
Section \ref{section.Notations} is devoted  to the introduction of the model and the notations.
Section \ref{section.resutats.perco} contains the main results of the paper, concerning percolation for the CRCM.
In Section \ref{section.resultat.phase.transition} we present a phase transition result for the Widom-Rowlinson model with random radii.
This result is a non-trivial extension of the result in \cite{chayes_kotecky} and \cite{georgii_haggstrom}, relying on an infinite-volume FK-representation.
In Section \ref{section_stochastic_domination} we introduce basic notions of stochastic domination and we prove, as a direct consequence of Theorem 1.1 of \cite{georgii_kuneth}, the absence of percolation for small activities in the case (C1) and the percolation for large activities in the case (C2).
Section \ref{section_preuve_theo1} is devoted to the proof of the percolation for large activities $z$ in the case (C1) and Section \ref{section_preuve_theo2} deals with the proof of the absence of percolation for small activities in the case (C2).
Finally Section \ref{section_preuve_resultats_intermediaires} contains a sketch of the proof of standard results involving the uniqueness of the unbounded cluster and the infinite-volume FK-representation.

\section{Notations and definitions}\label{section.Notations}
\subsection{State space and reference measure} 
For $d$ at least 2, we denote by $\Scb$ the space $\R^d \times \R^+$ endowed with the Borel $\sigma$-algebra.
A configuration  $\omega$ is a non-negative integer-valued measure on $S$, which can be represented as $\omega=\sum_{i\in I} \delta_{(x_i,R_i)}$ for a finite or infinite sequence $(x_i,R_i)_{i\in I}$, with finite mass on $\Lambda \times \R^+$ for every bounded subset $\Lambda$ of $\R^d$. 
The configuration set $\Omega$ is equipped with the classical $\sigma$-algebra $\Fcb$ generated by the counting variables $\omega \mapsto \omega(\Gamma)$ where $\Gamma$ is a bounded Borel subset of $\Scb$. 
The configuration $\omega$ restricted to a subset $\Lambda$ of $\R^d$ is defined by $\omega_{\Lambda}(.):=\omega (.\cap \Lambda \times \R^+)$ and $\Fcb _{\Lambda}$ is the sub $\sigma$-algebra of $\Fcb$ generated by the counting variables $\omega \mapsto \omega(\Gamma)$ where $\Gamma$ is a bounded subset of $\Lambda \times \R ^+$. 
We write $(x,R)\in \omega$ if $\omega(\{(x,R)\})>0$. 
To each configuration $\omega$ is associated the germ-grain structure 
$$
L(\omega)=\underset{(x,R)\in \omega}\bigcup B(x,R)
$$ 
where $B(x,R)$ is the Euclidean closed ball centred in $x$ with radius $R$.

For a positive $z$ and a probability measure $Q$ on $\R^+$, $\pi^{z,Q}$ denotes the distribution on $\Omega$ of the homogeneous Poisson point process with spatial intensity $z$ and with independent marks distributed by $Q$.
For $\Lambda \subseteq \R ^d$, $\pi^{z,Q}_{\Lambda}$ denotes the projection of $\pi^{z,Q}$ on $\Lambda \times \R^+$. 
The random closed set $L$ under the law $\pi^{z,Q}$ is the so-called Poisson Boolean model with intensity $z$ and radii distribution $Q$.

\subsection{Interaction}

For every configuration $\omega$, the connected components in $L(\omega)$ are defined as followed.
Two points $x,y \in \R ^d$ are in the same connected component of $L(\omega)$ if there is a finite number of points $(x_1,R_1), \dots ,(x_n,R_n)$ in $\omega$ such that
\begin{itemize}
\item $x \in B(x_1,R_1)$,
\item $y \in B(x_n,R_n)$,
\item $\forall i \in \{ 1, \dots ,n-1\}$, $B(x_i,R_i) \cap B(x_{i+1},R_{i+1}) \not = \emptyset$.
\end{itemize}
\begin{remark}
This connectivity is the usual one defined with the Gilbert graph.
However it is different from the topological one for some infinite configurations.
\end{remark}

The random cluster interaction between the particles in a finite configuration $\omega$ is given by the unnormalized density 
$ q^{ N_{cc}(\omega)} $,
where $N_{cc}(\omega)$ denotes the number of connected components of  $L(\omega)$. 
This density is well-defined only for finite configurations. 
As usual, for infinite configurations we define a local conditional density. 

\begin{definition}  
For a bounded subset $\Lambda$ of $\R^d$, the $\Lambda$-local number of connected components of a configuration (finite or infinite) is given by
\begin{equation}\label{definitionlimite}
N_{cc}^{\Lambda}(\omega) = \underset{\Delta \to \R^d}{\lim}
\left(  
N_{cc}(\omega_{\Delta}) - N_{cc}(\omega_{\Delta \setminus \Lambda})
\right) \nonumber
\end{equation}
where the limit, taken along any increasing sequence, is well defined, see \cite[Prop. 2.1]{dereudre_houdebert}. 
\end{definition}

\subsection{Continuum Random Cluster Model}
The continuum random cluster model is defined using standard DLR formalism which requires that the probability measure satisfies equilibrium equations based on Gibbs kernels, see equation \eqref{DLR}. 

\begin{definition}
A probability measure $P$ on $(\Ocb, \Fcb)$ is a continuum random cluster model for parameters $z$, $Q$ and $q$ (CRCM($z,Q,q$)) if it is stationary and if for every bounded subset $\Lambda$  of  $\R^d$ and every bounded measurable function $f$ we have
\begin{equation}\label{DLR}
\int_{\Ocb} f(\omega) P(d\omega)
=
\int_{\Ocb} \int_{\Ocb} f( \omega'_{\Lambda} + \omega_{\Lambda^c} )
\frac{q^{ N_{cc}^{\Lambda}( \omega'_{\Lambda} + \omega_{\Lambda^c} )}}
{Z_{\Lambda}(\omega_{\Lambda^c})}
\pi^{z,Q}_{\Lambda}(d\omega'_{\Lambda})
P(d\omega),
\end{equation}
where
$Z_{\Lambda}(\omega_{\Lambda^c})
:= \int_{\Ocb} q^{N_{cc}^{\Lambda}( \omega'_{\Lambda} + \omega_{\Lambda^c} )}
\pi^{z,Q}_{\Lambda}(d\omega'_{\Lambda})$ is the partition function which is assumed to be non-degenerate.

Equivalently, for $P$-almost every configuration $\omega$ the conditional law of $P$ given $\omega_{\Lambda^c}$ is absolutely continuous with respect to $\pi^{z,Q}_\Lambda$ with  density $q^{N_{cc}^{\Lambda}(.+\omega_{\Lambda^c}) }/ Z_{\Lambda}(\omega_{\Lambda^c}).$
\end{definition}

These equations, for every bounded $\Lambda$, are called DLR (Dobrushin, Lanford, Ruelle) equations. Existence and (non)-uniqueness of the CRCM are standard but complex questions of statistical mechanics. 
 In a recent paper \cite{dereudre_houdebert}, the existence was proved for a large class of parameters $(z,Q,q)$ such as the case $q<1$ or unbounded radii.
In this paper we are interested in the two following cases.
\begin{enumerate}
\item[(C1)] $q \geq 1$ and $Q$ satisfying $\int R^d Q(dR) <+\infty$.
\item[(C2)] $q<1$ and $Q$ satisfying $Q([0,R_0])=1$ for a given $R_0>0$.
\end{enumerate}
For both cases the existence of a CRCM was proved in \cite{dereudre_houdebert}.




\section{Results}
\subsection{Percolation results} \label{section.resutats.perco}
A configuration $\omega$ is said to percolate if $L(\omega)$ contains at least one unbounded connected component. 
We say a probability measure percolates (respectively does not percolate) if the probability of the percolation event $\{ \omega \text{ percolates} \}$ is $1$ (respectively $0$). 
The first natural question is about the number of infinite connected components, which is answered by  the following Theorem \ref{theo_uniquness_cc_infini}.

\begin{theorem}\label{theo_uniquness_cc_infini}
For every CRCM, almost surely the number of unbounded connected components is at most 1.
\end{theorem}

The proof of this result uses classical percolation techniques and is developed in Section \ref{section_preuve_resultats_intermediaires}.

The main question of the present paper is the classical question : "Does percolation occur?". 
One trivial case is when $Q=\delta_{0}$ and there is trivially no percolation since the CRCM is just a Poisson point process.
In the following this case is omitted.
The following theorems are the main results of the present paper.
These theorems prove the existence, for both cases (C1) and (C2), of a subcritical phase for small activities $z$ where no percolation occurs and a supercritical phase for large activities $z$ where percolation occurs.
\begin{theorem}\label{theoreme_case_1}
In the case (C1),
\begin{itemize}
\item there exists $z_0:=z_0(Q,q,d)$ positive such that for every $z<z_0$, no $CRCM(z,Q,q)$ percolates;
\item with the additional assumption $Q(\{0\})=0$, there exists $z_1:=z_1(Q,q,d)$ finite such that for every $z>z_1$, every $CRCM(z,Q,q)$ percolates.
\end{itemize}
\end{theorem}
\begin{theorem}\label{theoreme_case_2}
In the case (C2), 
\begin{itemize}
\item there exists $z_1:=z_1(Q,q,d)$ finite such that for every $z>z_1$, every $CRCM(z,Q,q)$ percolates;
\item there exists $z_0:=z_0(Q,q,d)$ positive such that for every $z<z_0$, no $CRCM(z,Q,q)$  percolates.
\end{itemize}
\end{theorem}
The proof of both Theorem \ref{theoreme_case_1} and Theorem \ref{theoreme_case_2} are based on stochastic domination techniques which enable us to compare the CRCM to more simple models, namely the Poisson Boolean model and the Bernoulli percolation model.
Proofs of those theorems are done in sections \ref{section_stochastic_domination}, \ref{section_preuve_theo1} and \ref{section_preuve_theo2}.
\begin{remark}
For both Theorem \ref{theoreme_case_1} and Theorem \ref{theoreme_case_2}, we get the existence of $0<z_0 \leq z_1< + \infty$.
The existence of a threshold, meaning that $z_0 = z_1$ is a natural and interesting question still open for the CRCM and many other models.
Indeed the only existing technique for this question is to prove stochastic monotonicity with respect to the parameter $z$.
This is proved for continuum models using the result of Georgii and K{\"u}neth \cite{georgii_kuneth}, which does not apply for the CRCM.
\end{remark}
\begin{remark}
The assumption $Q(\{0\})=0$ in Theorem \ref{theoreme_case_1} is an artefact of the proof developed in Section \ref{section_preuve_theo1}.
This assumption ensures the existence of all parameters introduced in the proof.
It is our belief that this theorem could be generalized to the optimal assumption $Q(\{0\})<1$.
\end{remark}

\subsection{Phase transition for the Widom-Rowlinson model}\label{section.resultat.phase.transition}
In this section the Widom-Rowlinson model (\cite{widom_rowlinson}) is introduced and a phase transition result is exhibited as a direct consequence of Theorem \ref{theoreme_case_1}.
This section gives a non-trivial generalisation of the results in \cite{chayes_kotecky} and \cite{georgii_haggstrom} to the case of unbounded radii.
It also gives a more direct proof relying on an infinite-volume Fortuin-Kasteleyn representation, Proposition \ref{proposition_FK_infini}.
Let $q$ be an integer larger than $1$.
The Widom-Rowlinson model is defined on the space $\Oc$ of coloured configuration $\tomega$ in 
$\Sc = \R ^d \times \R ^+ \times \{1, \dots q \}$.
The marks in $\{1, \dots q\}$ are called colours.
The definitions of the $\sigma$-field $\Fc$ and the coloured Poisson point process $\tpi^{z,Q,q}$ with uniform colour marks are the natural extension of the definitions in Section \ref{section.Notations} and are omitted.
Let $\A$ be the event of coloured configuration $\tomega$ such that any two balls of two different colours do not overlap.
\begin{definition}
A probability measure $\tP$ on ($\Oc, \Fc$) is a Widom-Rowlinson model of parameters $z,Q,q$ (WR($z,Q,q$)) if it is stationary and if, for every bounded set $\Lambda$ and every bounded measurable function $f$,
\begin{align}\label{equation_dlr_wr}
\int_{\Oc} f d \tP
=
\int_{\Oc} \int_{\Oc} f(\tomega'_{\Lambda} + \tomega_{\Lambda^c})
\frac{\1 _ {\A}(\tomega'_{\Lambda} + \tomega_{\Lambda^c})}
{\widetilde{Z}_{\Lambda}(\tomega_{\Lambda^c})}
\tpi_{\Lambda}^{z,Q,q}(d\tomega'_{\Lambda})
\tP(d\tomega).
\end{align}
\end{definition}
For $z>0$, $\int R^d Q(dR)< + \infty$ and $q$ a positive integer, the existence is well-known and can be seen as a consequence of the existence of the CRCM and the Fortuin-Kasteleyn representation Proposition \ref{proposition_FK_infini}.

\begin{theorem}\label{theoreme_phase_transition_wr}
For every integer $q$ larger than $1$ and every probability measure $Q$ satisfying $Q(\{0 \})=0$ and $\int R^d Q(dR)< \infty$, we have the existence of $\tilde{z}>0$ such that, for every $z>\tilde{z}$, there exists at least $q$ ergodic WR($z,Q,q$).
\end{theorem}
The proof of this theorem is based on an infinite-volume Fortuin-Kasteleyn representation given in Proposition \ref{proposition_FK_infini}.

\begin{proposition}\label{proposition_FK_infini}
Let $P$ be a CRCM($z/q,Q,q$).
We build a probability measure $\tP^1$ on $\Oc$ by colouring each finite connected component independently and uniformly over the $q$ colours $\{1, \dots , q \}$.
If it exists, the (unique) infinite connected component is assigned the colour $1$.

Then the measure $\tP^1$ is a WR($z,Q,q$). 
Moreover if $P$ is ergodic, the same is true for $\tP^1$.
\end{proposition}

The proof of Proposition \ref{proposition_FK_infini} is done in Section \ref{section_preuve_resultats_intermediaires}.
It relies on the GNZ equations satisfied by the Widom-Rowlinson model and the CRCM.

Using Theorem \ref{theoreme_case_1} we have for $z>qz_1$ that every CRCM($z/q,Q,q$)  percolates.
Choose an ergodic $P$ , which can be done since every CRCM($z/q,Q,q$) is a mixing of ergodic CRCM($z/q,Q,q$), see \cite{georgii_livre}.
From Proposition \ref{proposition_FK_infini} one can build $q$ ergodic WR $\tP^1, \dots, \tP^q$ which are distinct since the unbounded connected component does not have the same colour. 
This proves Theorem \ref{theoreme_phase_transition_wr}.

\section{Stochastic domination}
\label{section_stochastic_domination}
Stochastic domination is at the core of the proof of Theorem \ref{theoreme_case_1} and Theorem \ref{theoreme_case_2}.
Let us recall some basic definitions on stochastic domination. 
 We say that $\omega$ is smaller than $\omega'$, and we write $\omega \leq \omega'$, if $\omega(\Gamma) \leq \omega'(\Gamma)$ for every Borel set $\Gamma$ of $\Scb$. 
 An event $A$ in $\Fcb$ is said to be increasing if for every $\omega \in A$ and every $\omega' \geq \omega$, we have $\omega' \in A$. 
 One example of increasing event is the percolation event $\{\omega \in \Omega | \ \omega \text{ percolates} \}$. 
Finally if $P$ and $P'$ are two probability measures on $\Omega$, we say that $P'$ stochastically dominates $P$, writing $P \preceq P' $, if $P(A) \leq P'(A)$ for every increasing event $A$. 
One way of proving stochastic domination between Gibbs point processes is the Theorem 1.1 in \cite{georgii_kuneth} which relies on an inequality between the Papangelou intensities of the processes considered. 
This leads to the following proposition.
 \begin{proposition}
\label{proposition_domination_sto_poisson}
Let $P$ be a CRCM($z,Q,q$).
\begin{itemize}
\item If $q \geq 1$, then $P \preceq \pi^{qz,Q}$;
\item if $q <1$, then $\pi^{qz,Q} \preceq P$.
\end{itemize}
\end{proposition}
\begin{proof}
The Papangelou intensity of a CRCM($z,Q,q$) is $q^{1-k(X,\omega)}$, where $k(X,\omega)$ counts the number of connected components of $L(\omega)$ overlapping the ball $B(X)=B(x,R)$.
Since the function $k$ is non-negative, the Theorem 1.1 in \cite{georgii_kuneth} directly gives the result.
\end{proof}

Since the percolation event is increasing, using Proposition \ref{proposition_domination_sto_poisson} and the existence of  subcritical and supercritical phases for the Poisson Boolean model of radii law $Q$ satisfying $\int R ^ d Q(dR)<+ \infty$, see \cite{gouere_2008}, we get the following corollary proving the existence of a subcritical phase (respectively a supercritical phase) for the CRCM in the case (C1) (respectively (C2)).
\begin{corollary}
\label{coro_perco_gk}
Let $z_c(d,Q)$ be the percolation threshold of the Poisson point process $\pi^{z,Q}$, then
\begin{itemize}
\item in the case (C1), for every $z<z_c(d,Q)$, no CRCM($z,Q,q$) percolates;
\item in the case (C2), for every $z>z_c(d,Q)$, each CRCM($z,Q,q$) percolates.
\end{itemize}
\end{corollary}
\begin{remark}
The Corollary \ref{coro_perco_gk} gives the existence of the constant $z_0$ of Theorem \ref{theoreme_case_1} and the constant $z_1$ of Theorem \ref{theoreme_case_2}.
In some cases like deterministic radii we have $k((x,R),\omega) \leq \kappa (d)$ and Theorem 1.1 in \cite{georgii_kuneth} can be used to get the existence of the others constants.
This argument was already used in \cite{chayes_kotecky} and \cite{georgii_haggstrom}.
\end{remark}
In the general case, we do not have an upper bound for the quantity $k(X,\omega)$ and therefore Theorem 1.1 in \cite{georgii_kuneth} cannot be used to prove the rest of Theorem \ref{theoreme_case_1} and Theorem \ref{theoreme_case_2}. 

To this end we introduce the Proposition \ref{proposition_liggett_domination_sto} which provides a good stochastic domination between the law of a dependent family $(\xi_x)_{x \in \Z ^d}$ of random variables in $\{0,1\}$ and $\Pi^p$ the Bernoulli (with parameter $p$) product measure on $\{0,1 \}^{\Z^d}$.
This type of argument was already used to prove percolation results for continuum models, see \cite{coupier_dereudre}.

\begin{proposition}[Liggett, Schonmann and Stacey \cite{liggett_schonmann_stacey}]\label{proposition_liggett_domination_sto}
Let $(\xi_x)_{x \in \Z ^d}$ be a (dependent) family of $\{0,1\}$-valued random variables of joint law $\nu$.
Let $p \in [0,1]$ and assume that for every vertex $x$,
\begin{align}
\nu(\xi_x=1 | \xi_y, \lVert x-y \rVert _{\infty}>k)
\geq p \ a.s. \nonumber
\end{align}
where $\lVert . \rVert _{\infty}$ is the infinite norm on $\R^d$ and $k$ is a positive fixed constant.

Then $\Pi^{f(p)}$ is stochastically dominated by the distribution of $(\xi_x)_{x \in \Z^d}$, where $f$ is a deterministic function depending on $k$ such that $\underset{p \to 1}{\lim}f(p)=1$.
\end{proposition}

\section{Proof of Theorem \ref{theoreme_case_1}: percolation for large activities $z$} \label{section_preuve_theo1}
We assume here that conditions (C1) are satisfied: $q \geq 1$ and $Q$ satisfying $\int R^d Q(dR)< + \infty$. 
We are also making the extra assumption that $Q(\{0 \})=0$.
Let $P^z$ be a $CRCM(z,Q,q)$. 
The idea of the proof is to construct a family $(\xi_x)_{x \in \Z^d}$ such that the conditional probability as in Proposition \ref{proposition_liggett_domination_sto} is as large as we need and such that if the family $(\xi_x)$ percolates, the same is true for our model. 
To this end we look into the probability of covering small cubes.
Intuitively if the family of covered cubes percolates, then the same is true for the underlying configuration.
Let $R_1>0$ be such that $Q_{R_1} :=Q([0,R_1])<1/q$ and let $\Lambda$ be the cube $[-R_1/2\sqrt{d},R_1/2\sqrt{d} ]^d$. 
To look into the probability of covering the cube $\Lambda$, for a radius $R$ let $C_R$ be the event of configurations $\omega$ such that $\omega_{\Lambda}$ contains at least one ball of radius larger than $R$. 

We define the random variable $\xi:=\xi(\omega)$ equal to $1$ if $\Lambda \subseteq L(\omega_{\Lambda})$ and $0$ otherwise.
To use Proposition \ref{proposition_liggett_domination_sto} we need to prove that
\begin{align}
\label{convergence_cas1_but}
\underset{\omega_{\Delta^c} }{\inf}
P^z(\xi=1|\omega_{\Delta^c}) 
\underset{z \to \infty}{\longrightarrow} 1,
\end{align}
where $\Delta$ is a bounded set containing $\Lambda$ and which will be defined shortly.
By construction we have
\begin{align}
P^z(\xi=1|\omega_{\Delta^c}) 
\geq \nonumber
P^z(C_{R_1}|\omega_{\Delta^c}).
\end{align}
The goal now is to find a good bound for $P^z(C_{R_1}^c|\omega_{\Delta^c})$, uniform with respect to $\omega_{\Delta^c}$, which tends to $0$ when $z$ grows to infinity. 
To control the quantity $N_{cc}^{\Lambda}$, we need to introduce a "protective layer", hence the $\Delta$, and a good event $B_z$ on $\Delta\setminus \Lambda$ which is defined below.
We have
\begin{align} \label{equation_cas1_separation_en_deux}
P^z(C_{R_1}^c|\omega_{\Delta^c})
&=
P^z(C_{R_1}^c \cap B_z|\omega_{\Delta^c}) + 
P^z(C_{R_1}^c \cap B_z^c|\omega_{\Delta^c}) \nonumber
\\ & \leq 
\underbrace{P^z(B_z|\omega_{\Delta^c})}_{(a)}
+ 
\underbrace{P^z(C_{R_1}^c \cap B_z^c|\omega_{\Delta^c})}_{(b)}.
\end{align}
\subsection{Bound for the quantity $(b)$}
Thanks to the DLR equations \eqref{DLR} on $\Lambda$, we have
\begin{align}
P^z & (C_{R_1}^c \cap B_z^c|\omega_{\Delta^c})
\nonumber \\ &=
\int_{\Ocb} \int_{\Ocb} \1_{C_{R_1}^c}(\omega''_{\Lambda}) 
\1_{B_z^c}(\omega'_{\Delta \setminus \Lambda})
\frac{q^{N_{cc}^{\Lambda}
(\omega''_{\Lambda} + \omega'_{\Delta \setminus \Lambda} 
+ \omega_{\Delta^c} )}}
{Z_{\Lambda}(\omega'_{\Delta \setminus \Lambda}+\omega_{\Delta^c})}
\pi_{\Lambda}^{z,Q}(d\omega''_{\Lambda})
P^z(d\omega'_{\Delta}|\omega_{\Delta^c})
\nonumber \\ &\leq
\int_{\Ocb} \int_{\Ocb} \1_{C_{R_1}^c}(\omega''_{\Lambda}) 
\1_{B_z^c}(\omega'_{\Delta \setminus \Lambda})
\frac{q^{\# \omega''_{\Lambda}}}
{Z_{\Lambda}(\omega'_{\Delta \setminus \Lambda}+\omega_{\Delta^c})}
\pi_{\Lambda}^{z,Q}(d\omega''_{\Lambda})
P^z(d\omega'_{\Delta}|\omega_{\Delta^c})
\nonumber \\ & =
\int_{\Ocb} \1_{C_{R_1}^c}(\omega''_{\Lambda})
q^{\# \omega''_{\Lambda}}
\pi_{\Lambda}^{z,Q}(d\omega''_{\Lambda})
\int_{\Ocb} \frac{\1_{B_z^c}(\omega'_{\Delta \setminus \Lambda})}
{Z_{\Lambda}(\omega'_{\Delta \setminus \Lambda}+\omega_{\Delta^c})}
P^z(d\omega'_{\Delta}|\omega_{\Delta^c})
\nonumber \\ &\leq
e^{-z |\Lambda|(1-qQ_{R_1} )}
\int_{\Ocb} \frac{\1_{B_z^c}(\omega'_{\Delta \setminus \Lambda})}
{Z_{\Lambda}(\omega'_{\Delta \setminus \Lambda}+\omega_{\Delta^c})}
P^z(d\omega'_{\Delta}|\omega_{\Delta^c}),
\label{inegalite_cas1_R1}
\end{align}
where $\# \omega$ denotes the number of balls of a finite configuration $\omega$ and where the second line comes from the trivial  inequality
$N_{cc}^{\Lambda}(\omega) \leq \# \omega_{\Lambda}$. 
From \eqref{inegalite_cas1_R1} one can understand the condition $Q_{R_1}<1/q$.
Now we need to give a precise definition of $\Delta$ and $B_z$. 
To this end let $R_2$ be a positive real number such that $Q_{R_2}:=Q([0,R_2])>qQ_{R_1}$, which is possible by the choice of $R_1$.
We define $\Delta : = [-1-R_2-R_1/2\sqrt{d}, 1+R_2+R_1/2\sqrt{d} ]^d$ and we take $ \epsilon$ such that
$0<\epsilon < |\Lambda|(Q_{R_2} - q Q_{R_1}) / \ln q$.
Finally take $r>0$ such that $Q_r:=Q([0,r])< \frac{\epsilon}{q e |\Delta \setminus \Lambda|}$. 
The last thing to do is to define $B_z$ as the event of configurations having many small balls centred inside $\Delta \setminus \Lambda$,
\begin{align}
B_z
=
\{
\omega \in \Omega | \# \{ (x,R) \in \omega_{\Delta \setminus \Lambda} , R \leq r \}
\geq \lceil \epsilon z \rceil 
\},
\nonumber
\end{align}
where $\lceil. \rceil$ is the ceiling function.
The following Lemma \ref{lemme_minoration_Ncc} gives a good lower bound for $N_{cc}^{\Lambda}$.

\begin{lemma}\label{lemme_minoration_Ncc}
 For $\omega \in C_{R_2}^c \cap B_z^c$ we have 
$N_{cc}^{\Lambda} (\omega) \geq K-\lceil \epsilon z \rceil$ for a  constant $K$.
\end{lemma}

\begin{proof}[Proof of the lemma] 
We first realize that the worst case scenario occurs when $L(\omega_{\Lambda})$ has a single connected component which intersects many connected components of $L(\omega_{\Lambda^c})$. 
Since $\omega \in C_{R_2}^c$, the radii of the balls of $L(\omega_{\Lambda})$ are bounded by  $R_2$. 
Note that the number of connected components $L(\omega_{\Lambda^c})$ intersecting $L(\omega_{\Lambda})$ is smaller or equal to the number of disjoint balls in $L(\omega_{\Lambda^c})$ that can intersect $L(\omega_{\Lambda})$. 
By the choice of $\Delta$, a ball $B(x,R)$ in $L(\omega_{\Delta^c})$ which intersects $L(\omega_{\Lambda})$ satisfies $R \geq 1$ and hence $ | B(x,R) \cap \Delta| \geq c_1$ for a given positive constant $c_1$. 
In the same manner, any ball $B(x,R)$ of $L(\omega_{\Delta \setminus \Lambda})$ with $R >r$ satisfies $ | B(x,R) \cap \Delta| \geq c_2$.
So the number of such balls is at most $| \Delta|/\min (c_1,c_2) : = c$.
To get the bound on $N_{cc}^{\Lambda}(\omega)$, we just need to take into consideration the balls of $\omega_{\Delta \setminus \Lambda}$ with radius smaller than $r$, which number is not greater than $\lceil \epsilon z \rceil-1$ since $\omega \in B_z^c$. 
Hence the result holds by taking $K=2-c$.
\end{proof}
To get an upper bound for the integral in \eqref{inegalite_cas1_R1} we consider the conditional probability
\begin{align}
P^z & (C_{R_2}^c \cap B_z^c|\omega_{\Delta^c})
\nonumber \\ &=
\int_{\Ocb} \int_{\Ocb} \1_{C_{R_2}^c}(\omega''_{\Lambda}) 
\1_{B_z^c}(\omega'_{\Delta \setminus \Lambda})
\frac{q^{N_{cc}^{\Lambda}
(\omega''_{\Lambda} + \omega'_{\Delta \setminus \Lambda} + \omega_{\Delta^c} )}}
{Z_{\Lambda}(\omega'_{\Delta \setminus \Lambda}+\omega_{\Delta^c})}
\pi_{\Lambda}^{z,Q}(d\omega''_{\Lambda})
P^z(d\omega'_{\Delta}|\omega_{\Delta^c})
\nonumber \\ & \geq
\int_{\Ocb} \int_{\Ocb} \1_{C_{R_2}^c}(\omega''_{\Lambda}) 
\1_{B_z^c}(\omega'_{\Delta \setminus \Lambda})
\frac{q^{K-\lceil \epsilon z \rceil}}
{Z_{\Lambda}(\omega'_{\Delta \setminus \Lambda}+\omega_{\Delta^c})}
\pi_{\Lambda}^{z,Q}(d\omega''_{\Lambda})
P^z(d\omega'_{\Delta}|\omega_{\Delta^c})
\nonumber \\ &=
q^{K-\lceil \epsilon z \rceil}
\int_{\Ocb} \1_{C_{R_2}^c}(\omega''_{\Lambda}) 
\pi_{\Lambda}^{z,Q}(d\omega''_{\Lambda})
\int_{\Ocb} \frac{\1_{B_z^c}(\omega'_{\Delta \setminus \Lambda})}
{Z_{\Lambda}(\omega'_{\Delta \setminus \Lambda}+\omega_{\Delta^c})}
P^z(d\omega'_{\Delta}|\omega_{\Delta^c})
\nonumber \\ & =
q^{K-\lceil \epsilon z \rceil}
e^{-z |\Lambda| (1- Q_{R_2})}
\int_{\Ocb} \frac{\1_{B_z^c}(\omega'_{\Delta \setminus \Lambda})}
{Z_{\Lambda}(\omega'_{\Delta \setminus \Lambda}+\omega_{\Delta^c})}
P^z(d\omega'_{\Delta}|\omega_{\Delta^c}).
\label{inegalite_cas1_R2}
\end{align}

From inequalities \eqref{inegalite_cas1_R1} and \eqref{inegalite_cas1_R2} we get
\begin{align}
P^z(C_{R_1}^c \cap B_z^c|\omega_{\Delta^c})
& \leq
q^{\lceil \epsilon z \rceil -K}
e^{-z |\Lambda| ( Q_{R_2}-qQ_{R_1} )}
\nonumber \\ & \leq
q^{-K+1} e^{-z |\Lambda| ( Q_{R_2}-qQ_{R_1} )} e^{\epsilon z \ln q }.
\label{inegalite_cas1_final1}
\end{align}
By the choice of $R_1$, $R_2$ and $\epsilon$ we have
$|\Lambda|(Q_{R_2} - q Q_{R_1})-\epsilon \ln q >0 $ and therefore the right hand side of \eqref{inegalite_cas1_final1} decreases to $0$ when $z$ grows to infinity.
\subsection{Bound for the quantity $(a)$}
Now we need to control the quantity $(a)$ in \eqref{equation_cas1_separation_en_deux}. 
Using Theorem 1.1 of \cite{georgii_kuneth}, the probability measure $\pi_{\Delta}^{qz,Q}$ stochastically dominates $P^z(.|\omega_{\Delta^c})$, and since the event $B_z$ is increasing, we have
\begin{align}
P^z(B_z|\omega_{\Delta^c})
& \leq
\pi_{\Delta}^{qz,Q}(B_z)
\nonumber \\ & =
e^{-z q |\Delta \setminus \Lambda| Q_r}
\underset{k \geq \lceil \epsilon z \rceil}{\sum}
\frac{(z q |\Delta \setminus \Lambda| Q_r )^k}{k!}
\nonumber \\ & \leq 
\frac{(z q |\Delta \setminus \Lambda| Q_r )^{\lceil \epsilon z \rceil } }{\lceil \epsilon z \rceil!}
\underset{z \to \infty}{\sim}
\frac{
\exp \left( \lceil \epsilon z \rceil \ln \left(
\frac{z q |\Delta \setminus \Lambda | Q_r e}{\lceil \epsilon z \rceil}
\right)  \right)
}
{\sqrt{2 \pi \lceil \epsilon z \rceil}}
,
\label{inegalite_cas1_final2}
\end{align}
where the last line comes from the Lagrange inequality and the Stirling formula.
By the choice of the parameters, we have
\begin{align}
\frac{z q |\Delta \setminus \Lambda | Q_r e}{\lceil \epsilon z \rceil}
& \leq
\frac{z q |\Delta \setminus \Lambda | Q_r e}{\epsilon z}
<1, \nonumber
\end{align}
hence the right-hand side of \eqref{inegalite_cas1_final2} converges to $0$ when $z$ grows to infinity, and the convergence \eqref{convergence_cas1_but} is proven.
\subsection{Construction of the dependent family $(\xi_x)_{x \in \Z^d}$}
To each $x \in \Z^d$ we associate the small cube 
$\Lambda_x:= \frac{R_1}{\sqrt{d}}x \oplus \Lambda$ and the large cube
$\Delta_x :=  \frac{R_1}{\sqrt{d}}x \oplus \Delta $, where $\oplus$ stands for standard Minkowski sum.

We define $\xi_x$ as we did for $\xi := \xi_0$, meaning that 
$\xi_x(\omega)$ is equal to $1$ if $\Lambda_x \subseteq L(\omega_{\Lambda_x})$, and $0$ otherwise.
Finally take $k$ such that if $\lVert x-y \rVert _{\infty} >k$, then $\Delta_x \cap \Delta_y = \emptyset$.
For example $k=2 + \frac{2 \sqrt{d}R_2}{R_1}+ \frac{2 \sqrt{d}}{R_1}$ works.
For every $x \in V$, since $P^z$ is stationary  we have
\begin{align}
P^z(\xi_x =1 | \xi_y, \lVert x-y \rVert _{\infty}  >k)
&=
P^z(\xi=1 | \xi_y, \lVert y \rVert _{\infty} >k)
\nonumber 
\\ &=
\Esp_{P^z} [ P^z(\xi=1 | \Fcb_{\Delta^c})|
\xi_y, \lVert y \rVert _{\infty} >k]
\nonumber 
\\ & \geq 
\underset{\omega_{\Delta^c}}{\inf} P^z(\xi=1 | \omega_{\Delta^c}).
\nonumber
\end{align}
But using the convergence \eqref{convergence_cas1_but}, we have for any $p \in ]0,1[$ the existence of $z_1(p)$ such that, for every $z>z_1(p)$ and every $x$
\begin{align}
P^z(\xi_x=1 | \xi_y, \lVert y \rVert _{\infty} >k)
\geq
p \text{ a.s.}
\nonumber
\end{align}

Using Proposition \ref{proposition_liggett_domination_sto}, $\Pi^{f(p)}$ is stochastically dominated by the law of $(\xi_x)_{x \in \Z^d}$. 
The parameter $p$ can be chosen such that $f(p)$ is larger than any given percolation threshold.
But it is clear that if the family $(\xi_x(\omega))_{x \in \Z^d}$ percolates, with respect to the cubic lattice, the same is true for the configuration $\omega$.
Therefore if the activity $z$ is large enough, then $f(p)$ is larger than the percolation threshold of $\Pi^{f(p)}$ and $P^z$ percolates.
Theorem \ref{theoreme_case_1} is proved.
%
%
\begin{remark}
The assumption $Q(\{0\})=0$ was used in the proof to ensure that constants $R_1$ and $r$ are well defined.
It seems possible to do the same construction for radii law $Q$ having a small atom in $0$, but since the construction of $r$ depends on all the others constants introduced in the proof, it is impossible to derive from the proof presented a general condition in that case.
\end{remark}

\section{Proof of Theorem \ref{theoreme_case_2}: absence of percolation for small activities $z$} \label{section_preuve_theo2}
We assume here that conditions (C2) are satisfied: $q<1$ and there exists $R_0>0$ such that $Q([0,R_0])=1$. 
Let $P^z$ be a $CRCM(z,Q,q)$.
Without loss of generality we are making the proof in the specific case $R_0=1$.
The idea is, as in the proof of Theorem \ref{theoreme_case_1}, to build a good family $(\xi_x)_{x \in \Z^d}$ in order to apply Proposition \ref{proposition_liggett_domination_sto}.

Let $\Lambda=[-0.5,0.5]^d$ and $\Delta=[-8,8]^d$. 
We define $\xi : =\xi(\omega)$ equal to $1$ when $L(\omega) \cap \Lambda = \emptyset$ and $0$ otherwise. 
As the radii are bounded by $1$, $\xi$ depends only on the restricted configuration $\omega_{\Delta}$. 

We want to prove that
\begin{align}
\underset{\omega_{\Delta^c}}{\inf}P^z & (\xi=1 | \omega_{\Delta^c})
\underset{z \to 0}{\longrightarrow}
1.
\label{limite_cas2_but}
\end{align}
For a configuration $\omega$, let $N_{\xi}(\omega)$ denote the number of connected components of $L(\omega_{\Delta})$ which overlap $\Lambda $. The random variables $\xi$ and $N_{\xi}$ are strongly related since $\xi=1$ if and only if $N_{\xi}=0$.

\begin{lemma}\label{lemme_cas2_controle_P(B)}
There exists $\alpha>0$ such that
\begin{align}
\nonumber
\int_{\Ocb} N_{\xi}(\omega )
P^z(d \omega_{\Delta}|\omega_{\Delta^c})
\leq z |\Delta| q^{1-\alpha} \ a.s.
\end{align}
\end{lemma}
This lemma is proved at the end of the section.
First let us see how Lemma \ref{lemme_cas2_controle_P(B)} leads to the wanted result.
We have
\begin{align}
P^z  (\xi=0 | \omega_{\Delta^c})
& = \nonumber
P^z  (N_{\xi}>0 | \omega_{\Delta^c})
=
P^z  (N_{\xi}>0.5 | \omega_{\Delta^c})
\\ & \leq
2 z  |\Delta| q^{1-\alpha}, \nonumber
\end{align}
where the last inequality comes from the conditional Markov inequality and Lemma \ref{lemme_cas2_controle_P(B)}.
From this bound which is not depending in $\omega_{\Delta^c}$, the uniform convergence \eqref{limite_cas2_but} follows.

Now consider the family of variables $(\xi_x)_{x \in \Z^d}$ defined as $\xi: =\xi_0$, meaning that $\xi_x(\omega)=1$ if $L(\omega) \cap \Lambda_x = \emptyset$, where $\Lambda_x=x \oplus \Lambda$. 
Take $k=17$ and let $\Delta_x :=x \oplus  \Delta$.
For every vertex $x$ we have
\begin{align}
P^z(\xi_x=1 | \xi_y, \lVert x-y \rVert _{\infty}>k)
&=
P^z(\xi=1 | \xi_y, \lVert y \rVert _{\infty}>k)
\nonumber \\ &=
\Esp_{P^z} [ P^z(\xi=1 | \Fcb_{\Delta^c})|
\xi_y, \lVert y \rVert _{\infty}>k]
\nonumber \\ & \geq 
\underset{\omega_{\Delta^c}}{\inf} P^z(\xi=1 | \omega_{\Delta^c}),
\nonumber
\end{align}
and using the convergence \eqref{limite_cas2_but}, we have for any $p \in ]0,1[$ the existence of $z_0(p)$ such that, for every $z<z_0(p)$ and every $x$,
\begin{align}
P^z(\xi_x=1 | \xi_y, \lVert y \rVert _{\infty}>k)
\geq
p \text{ a.s.}
\nonumber
\end{align}

Using Proposition \ref{proposition_liggett_domination_sto}, $\Pi^{1-f(p)}$ stochastically dominates the law of $(1 - \xi_x)_{x \in \Z^d}$. 
The parameter $p$ can be chosen such that $1-f(p)$ is lower than any given percolation threshold.

As in Section \ref{section_preuve_theo1} the absence of percolation of the family $(1-\xi_x(\omega))_{x \in \Z^d}$, with respect to the cubic latice, leads to the absence of percolation in the configuration $\omega$.
Therefore for small activities $z$, $P^z$ does not percolate.

%
%

\begin{proof}[Proof of Lemma \ref{lemme_cas2_controle_P(B)}]
We use the GNZ equations satisfied by the (conditional) probability measure $P^z(.|\omega_{\Delta^c})$.
\begin{lemma}[GNZ equations]
\label{lemme_gnz_conditionnel_crcm}
For every bounded measurable function $F$,
\begin{align}
\int \underset{X \in \omega_{\Delta}}{\sum}
F(X,\omega - \delta_X)& P(d\omega_{\Delta}|\omega_{\Delta^c})
\nonumber \\ &=
z \int \int_{\Delta \times \R^+}
F(X,\omega) q^{1-k(X,\omega)}
 m(dX)  P(d\omega_{\Delta}|\omega_{\Delta^c}),
\nonumber
\end{align}
where $X=(x,R)$ is a marked point, where 
$m(dX)=\mathcal{L}^d(dx) \ Q(dR)$ with $\mathcal{L}^d$ the Lebesgue measure on $\R^d$ and where $k(X,\omega)$ is the function defined in the proof of Proposition \ref{proposition_domination_sto_poisson} as the number of connected components of $L(\omega)$ which overlap the ball $B(X)=B(x,R)$.
In the integrals $\omega$ stands for $\omega_{\Delta} + \omega_{\Delta^c}$.
\end{lemma}
The proof of this lemma is exactly the same as the proof of the well-known Slivnyak-Mecke formula and is omitted.
To use the GNZ equations let us define a function $F$ such that $F((x,R),\omega)$ is equal to $1$ if the following conditions are fulfilled and $0$ otherwise.
\begin{enumerate}
\item $x \in \Delta$.

\item The connected component of $B(x,R)$ in $L(\omega_{\Delta } +\delta_{(x,R)})$ intersects $\Lambda$.

\item $B(x,R)$ is one of the balls of its connected component of $L(\omega_{\Delta} +\delta_{(x,R)})$ which minimize the quantity $k((x,R),\omega)$.
\end{enumerate}

By the GNZ equations applied to the bounded function 
$F(X,\omega)q^{k(X,\omega)-1}$, we have
\begin{align}
\int_{\Ocb} \underset{(x,R) \in \omega_{\Delta}}{\sum}  &
F((x,R),\omega-\delta_{(x,R)}) q^{k((x,R),\omega-\delta_{(x,R)})-1}
P^z(d\omega_{\Delta}|\omega_{\Delta^c})
\nonumber \\ & =
z \int_{\Ocb} \int_{\Delta\times \R^+}
F((x,R),\omega) 
m(dx,dR)
P^z(d\omega_{\Delta}|\omega_{\Delta^c})
\leq
z |\Delta|.
\label{egalite_gnz_comptage_cc}
\end{align}

In the following we find an upper bound $\alpha$ for the quantity 
$k((x,R),\omega- \delta_{(x,R)})$ when $F((x,R),\omega-\delta_{(x,R)})=1$.
To this end consider a ball $B(x,R)$ and let us count the maximum number of disjoint balls of radius not smaller than $R/2$ overlapping $B(x,R)$.
By a volume argument, this quantity is no more than $4^d$.
To obtain this value we just look at the intersection of each disjoint ball with the ball $B(x,2R)$.
The volume of each intersection is at least $(R/2)^d  v_d$ where $v_d$ is the volume of the unit ball in dimension $d$.
This bound is not sharp but is enough for the following, since it does not depend on $R$.

By contradiction assume that in $L(\omega_{\Delta })$ there is a connected component $L(\mathcal{C})$ (with this notation $\mathcal{C}$ is a restriction of $\omega_{\Delta}$) which overlaps $\Lambda$ and such that 
$\underset{(x,R)\in \mathcal{C}}{\min}k((x,R),\omega -\delta_{(x,R)} ) \geq \alpha$.

Let $(x_1,R_1) \in \mathcal{C}$ such that $B(x_1,R_1)$ overlaps  $\Lambda $. 
Such a ball exists by hypothesis on $\mathcal{C}$ and we have $R_1\leq 1$ and the distance between $x_1$ and $\Lambda$ is no more than $1$.

By hypothesis $k((x_1,R_1),\omega - \delta_{(x_1,R_1)}) \geq \alpha$, therefore $B(x_1,R_1)$ is connected to at least one ball of radius smaller than $R_1/2\leq 1/2$. 
Let $B(x_2,R_2)$ be such a ball.
The distance between $x_2$ and $\Lambda$ is no more than $5/2$, which means that $(x_2,R_2) \in \mathcal{C}$.
By the same argument, we can construct a sequence $(x_n,R_n)_n$ with $R_n \leq 2^{-n+1}$ and the distance between $x_n$ and $\Lambda$ is bounded by some quantity increasingly converging to $4$.
Therefore each $(x_n,R_n)$ is in $\mathcal{C}$. 
But this is impossible because $\mathcal{C}$ is a finite configuration.
So in the left-hand side of \eqref{egalite_gnz_comptage_cc}, each $(x,R)$ such that $F$ is equal to one satisfies $k((x,R),\omega-\delta_{(x,R)})\leq \alpha$.

With this we get from \eqref{egalite_gnz_comptage_cc} the following inequality

\begin{align}
\int_{\Ocb}
\underset{(x,R) \in \omega_{\Delta}}{\sum}
F((x,R),\omega - \delta_{(x,R)})
P^z(d\omega_{\Delta}|\omega_{\Delta^c})
\leq
z |\Delta| q^{1-\alpha}.
\nonumber
\end{align}
But
$ \underset{(x,R) \in \omega_{\Delta}}{\sum}
F((x,R),\omega - \delta_{(x,R)})$ 
dominates $N_{\xi}(\omega)$ (since a connected component can be counted several times if it has several "minimizing balls"), and we get the result.
\end{proof}

\section{Annex}\label{section_preuve_resultats_intermediaires}
\subsection{Proof of Theorem \ref{theo_uniquness_cc_infini}}
Let $P$ be a CRCM($z,Q,q$).
Since $P$ is a mixture of ergodic CRCM($z,Q,q$), see \cite{georgii_livre}, we can assume without loss of generality that $P$ is ergodic.
We want to prove Theorem \ref{theo_uniquness_cc_infini}, which states that $P$ has almost surely at most one infinite connected component.
This is a classical question of percolation theory which was treated for many models.
The proof involves a positive correlation inequality between specific events.
In the literature this positive correlation inequality was proved using independence property (\cite{meester_roy}), FKG inequalities  or finite energy property (\cite{burton_keane}).
Our model does not satisfy the independence property, and the FKG inequalities are not proved.
The next lemma gives a (weak) continuum version of the finite energy property.
\begin{lemma}\label{lemme_local_modification}
Let $\Lambda$ be a bounded subset of $\R^d$. 
Let $A$ be a event of $\Fcb_{\Lambda}$ such that $\pi_{\Lambda}^{z,Q}(A)>0$.
Let $B$ be an event of $\Fcb_{\Lambda^c}$ such that $P(B)>0$.
Then
\begin{align}
P(A \cap B)>0. \nonumber
\end{align}
\end{lemma}
\begin{proof}
Since $P$ satisfies the DLR equations \eqref{DLR}, we have
\begin{align} 
P(A \cap B)
=
\int_{\Ocb} \int_{\Ocb} \1_A(\omega'_{\Lambda})
\1_B(\omega_{\Lambda^c})
\frac{q^{N_{cc}^{\Lambda}(\omega'_{\Lambda}+\omega_{\Lambda^c})}}
{Z_{\Lambda}(\omega_{\Lambda^c})}
\pi_{\Lambda}^{z,Q,q}(d\omega'_{\Lambda})
P(d\omega). \nonumber
\end{align}
The function integrated is positive, therefore the integral is positive and the lemma is proven.
\end{proof}

With Lemma \ref{lemme_local_modification}, one can carry out the rest of this classical proof of percolation theory.
We refer to \cite[section 3.6]{meester_roy} for more details.

\subsection{Proof of Proposiotion \ref{proposition_FK_infini}}
In this section $P$ is a CRCM($z/q,Q,q$) and $\tP^1$ is the associated coloured measure as in Proposition \ref{proposition_FK_infini}.
The colouration kernel is denoted by $C^1$, meaning that
\begin{align}
\nonumber
\tP^1(d\tomega)
=
C^1(d\tomega|\omega) P(d\omega).
\end{align}

To prove that $\tP^1$ is a WR($z,Q,q$), we are going to use the GNZ equations satisfied by the CRCM and the WR.
\begin{lemma}
\label{lemme_eq_gnz_crcm_wr}
A probability measure $P$ is a CRCM($z,Q,q$) if and only if for every bounded measurable function $G$ we have
\begin{align}
\int_{\Ocb} \underset{X \in \omega}{\sum}
G(X,\omega- \delta_X) P(d\omega)
=z \int_{\Ocb} \int_{\Scb} G(X,\omega) q^{1- k(X,\omega)} m(dX) P(d\omega),
\nonumber
\end{align}
where $m$ and $k$ are defined in Lemma \ref{lemme_gnz_conditionnel_crcm}.

A probability measure $\tP$ is a WR($z,Q,q$) if and only if for every bounded measurable function $F$ we have
\begin{align}
\int_{\Oc} \underset{\tX \in \tomega}{\sum}
F(\tX,\tomega- \delta_{\tX}) \tP(d\tomega)
=z \int_{\Oc} \int_{\Sc} F(\tX,\tomega) 
\1_{\A}(\tomega + \delta_{\tX}) \tm(d\tX) \tP(d\tomega),
\nonumber
\end{align}
where $\tX=(x,R,i)$ is coloured marked point and where
$\tm= \mathcal{L}^d \otimes Q \otimes \mathcal{U}_q $, with $\mathcal{U}_q$ the uniform measure on the set $\{1, \dots , q\}$.
\end{lemma}
The proof of this lemma is done for general Gibbs interactions in \cite{nguyen_zessin}.

By the definition of $\tP^1$ we have
\begin{align}\label{equation_preuve_FK_1}
& \int_{\Oc} \int_{\Sc}  z F(\tX,\tomega) \1_{\A}(\tomega + \delta_{\tX})
\tm(d\tX) \tP^1(d\tomega)
\nonumber \\ &=
\int_{\Ocb} \int_{\Scb} \underbrace{\left(
\int_{\Oc} \underset{k=1 \dots q}{\sum} 
F((X,i),\tomega) \1_{\A}(\tomega + \delta_{(X,i)} ) C^1(d\tomega|\omega)
\right)}_{G(X,\omega)}
 \frac{z}{q} m(dX) P(d\omega).
 \nonumber
\end{align}
 Since $P$ is satisfying the GNZ equations of the CRCM($z/q,Q,q$), we have
\begin{align}
 \int_{\Oc} \int_{\Sc}  z F(\tX,\tomega) \1_{\A}(\tomega + \delta_{\tX})
& \tm(d\tX) \tP^1(d\tomega)
 \nonumber \\ &=
\int_{\Ocb} \underset{X \in \omega}{\sum} 
G(X,\omega - \delta_{X})
q^{k(X,\omega-\delta_X)-1}
P(d\omega). \nonumber
\end{align}

In $G(X,\omega - \delta_X)$ the indicator function is equal to $1$ if
\begin{itemize}
\item The ball $B(X)$ is connected only to finite connected components of $L(\omega - \delta_X)$. 
In that case for each colour $i$ of $X$, there is only one good colouration, among the $q^{k(X,\omega - \delta_X)}$ possible colourations, of the $k(X,\omega - \delta_X)$ connected components.

\item The ball $B(X)$ is connected to the only infinite connected component of $L(\omega - \delta_X)$. 
In that case $X$ can only take the colour $1$ and there is only one good colouration, among the $q^{k(X,\omega - \delta_X)-1}$ possible colourations since the infinite connected component is coloured $1$.
\end{itemize}

In both cases we find that

\begin{align}
G(X,\omega - \delta_{X}) q^{k(X,\omega-\delta_X)-1}
=
\int F(\tX,\tomega - \delta_{\tX}) C^1(d\tomega|\omega),
\nonumber
\end{align}
which closes the proof of Proposition \ref{proposition_FK_infini}.
\vspace{2cm}

{\it Acknowledgement:} This work was supported in part by the Labex CEMPI  (ANR-11-LABX-0007-01). 
\bibliographystyle{plain}
\bibliography{biblio}
\end{document}